\documentclass{article}

\usepackage{amsmath}
\usepackage{amsfonts}
\usepackage{amsthm}
\usepackage{amssymb}
\usepackage{graphicx}
\usepackage{color}
\usepackage{verbatim} 
\usepackage{hyperref}
\usepackage{enumerate}
\usepackage{placeins}
\usepackage[utf8]{inputenc}
\usepackage{xcolor,graphicx}
\usepackage{tikz}
\usepackage{lmodern}
\usepackage[margin=1.5in]{geometry}
\usepackage{wrapfig}
\usepackage{floatrow}
\usepackage{soul}
\usepackage{filecontents}

\theoremstyle{plain}
\newtheorem{theorem}{Theorem}

\newtheorem{corollary}[theorem]{Corollary}
\newtheorem{lemma}[theorem]{Lemma}

\theoremstyle{definition}

\newtheorem{definition}{Definition}
\newtheorem{example}{Example}

\AtBeginDocument{}
\begin{document}

\title{P-positions in Modular Extensions to Nim}
\author{Tanya Khovanova \and Karan Sarkar}
\maketitle

\large

\begin{abstract}
In this paper, we consider a modular extension to the game of Nim, which we call $m$-Modular Nim, and explore its optimal strategy. In $m$-Modular Nim, a player can either make a standard Nim move or remove a multiple of $m$ tokens in total. We develop a winning strategy for all $m$ with $2$ heaps and for odd $m$ with any number of heaps. 
\end{abstract}

\section{Introduction}

Nim forms the foundation of the mathematical study of two-player strategy games. In his landmark $1901$ paper, \textit{Nim, a game with a complete mathematical theory}, Charles L. Bouton provided a solution to the game of Nim, essentially founding the field of Combinatorial Game Theory \cite{Nim}.

Since Bouton's discovery, many extensions or variants of Nim have been explored. Some variations that come to mind are Wythoff's Game, Poker Nim and Kayles. These variations often yield winning strategies that bare little resemblance to that of Nim \cite{WinningWays}.

Interestingly, very few if any of these variations use moves predicated upon modular congruence. In this paper, we explore a modular extension to Nim, which we call $m$-Modular Nim, in which moves are indeed predicated upon modular congruence are added to the traditional Nim moves. 

We start this paper with preliminaries in Section~\ref{prelim}. In Section~\ref{mModularNim}, we introduce the game of Modular Nim which is similar to Nim but in addition to Nim moves it allows players to remove a positive multiple of $m$ tokens total from the position. 

Section~\ref{2HeapOdd} considers $2$ heap Modular Nim for odd modular bases. Starting with an example for $m = 3$, we prove that the number of P-positions is finite and equal to $m$ for odd $m$. In Section~\ref{2HeapEven}, we expand our result to even values of $m$ by observing a self-similar structure in the set of P-positions. In Section~\ref{sec:explicit} we describe the P-positions explicitly.

In Section~\ref{sec:manyheaps} we describe P-positions in $m$-Modular Nim for any number of heaps and odd $m$.

\section{Preliminaries}\label{prelim}

We will be investigating the broad field of Combinatorial Game Theory (CGT). Roughly speaking, CGT concerns the study of winning strategies in two-player perfect information games. Our exploration of this large topic begins with some basic yet essential definitions \cite{WinningWays}.

\begin{definition}An \emph{impartial combinatorial game} is a two-player game where each player has both the same moves available at each and every point in the game and a complete set of information about the game and the potential moves.  \end{definition} 

This implies that no randomness such as rolling dice can exist.

\begin{definition}In \emph{normal play}, the first player unable to move is declared the loser. \end{definition} 

\begin{definition} We will call a position a \emph{terminal position} if no moves may be made from it. \end{definition}

In general, impartial combinatorial games are analyzed using the notion of P-positions and N-positions. This system of notation allows for games to be solved from the bottom up.

\begin{definition}A \emph{P-position} is a position from which the \emph{previous} player will win given perfect play. The set of P-positions is denoted as $\mathcal{P}$. \end{definition}

We can observe that all terminal positions are P-positions.

\begin{definition}An \emph{N-position} is a position from which the \emph{next} player will win given perfect play. The set of N-positions is denoted as $\mathcal{N}$.\end{definition}

Any position in the game is either a P-position or an N-position. All the moves from any P-position lead to an N-position. On the other hand, from any N-position, there exists some move to a P-position. These observations motivate the following theorem \cite{LessonsInPlay}.

\begin{theorem}Suppose that the positions of a finite impartial game can be partitioned into disjoint sets $\mathcal{A}$ and $\mathcal{B}$ such that:

\begin{enumerate}
\item Every move from a position in $\mathcal{A}$ is to a position in $\mathcal{B}$.
\item Every move from a position in $\mathcal{B}$ has at least one move to a position in $\mathcal{A}$.
\item All terminal positions are elements of $\mathcal{A}$.
\end{enumerate}

Then $\mathcal{A} = \mathcal{P}$ and $\mathcal{B} = \mathcal{N}$.
\end{theorem}

\subsection{Nim}

Nim is the most fundamental impartial combinatorial game. It is one of the earliest \emph{take-away games} and is testament to the complexity that can arise from simple rules \cite{LessonsInPlay}.

\begin{definition}In the game of \emph{Nim}, each position consists of a set of heaps of tokens. In a move, a player must remove a positive number of tokens from a single heap. \end{definition}

In a Nim-like game, denote positions with $h$ heaps as ordered $h$-tuples. To describe the set of P-positions in Nim, we need to define the bitwise XOR operation.

\begin{definition} The \textit{bitwise XOR} of two numbers is calculated by writing both numbers in binary and adding them without carrying over. We will use the $\oplus$  symbol to denote the bitwise XOR operation. 
\end{definition}

The set of P-positions in Nim is well understood and summarized by the following theorem \cite{Nim}.

\begin{theorem}[Bouton's Theorem]In Nim, $(a_1, \ldots, a_n) \in \mathcal{P}$ if and only if \[\bigoplus_{i=1}^n a_i = 0.\] \end{theorem}

\section{$m$-Modular Nim}\label{mModularNim}

We will now introduce a natural extension to Nim which we will subsequently discuss in detail. Our game involves loosening the restrictions on Nim moves with conditions based on modular congruence.

\begin{definition}In the game of \emph{$m$-Modular Nim}, each position consists of a set of heaps of tokens, like in Nim. However, we have two types of moves:

\begin{enumerate}[Type I.]
\item Remove a positive number of tokens from a single heap.
\item Remove $km$ tokens total where $k$ is a positive integer.
\end{enumerate}
\end{definition}

In our analysis of $m$-Modular Nim, we introduce some additional notation and positional functions for convenience.

\begin{definition}Let the \emph{heap-sum} of a position $A$ be the total number of tokens. We denote it as $|A|$. \end{definition}

We introduce a partial order on the set of positions to allow ourselves to speak more concisely about important concepts.

\begin{definition}If $A = \left(a_1, a_2, a_3, \ldots, a_k\right)$ and $B = \left(b_1, b_2, b_3, \ldots, b_k\right)$ are positions in $m$-Modular Nim such that $a_i \geq b_i$ for all integers $1 \leq i \leq k$, we say that $A$ \emph{dominates} $B$. Position $A$ \emph{strictly dominates} $B$ if $A$ dominates $B$ and $A$ is not equal to $B$, that is, there exists $i$ such that $a_i > b_i$. We denote domination as $A \succeq B$ and strict domination as $A \succ B$. 

Moreover if all members of set $\mathcal{S}$ dominate all members of set $\mathcal{T}$, we say that $\mathcal{S}$ dominates $\mathcal{T}$ or $\mathcal{S} \succeq \mathcal{T}$. Similarly, if all members of $\mathcal{S}$ strictly dominate all members of $\mathcal{T}$, we say that $\mathcal{S}$ strictly dominates $\mathcal{T}$ or $\mathcal{S} \succ \mathcal{T}$. \end{definition}

If a P-position $A$ dominates a P-position $B$, one might expect that then there exists an optimal game when $A$ occurs as a position in the game before $B$. This is true for a $2$-heap game, but is not true for a game with more heaps. 

For example, consider a 3 heap game of $4$-Modular Nim. The P-position $(1,2,2)$ dominates the P-position $(0,1,1)$, but the latter position can not be reached from the former position in any optimal play.

\begin{lemma}\label{DominateTypeII}
A Type II move from position $A$ to position $B$ exists, if and only if $|A| \equiv |B| \pmod{m}$ and $A \succ B$. \end{lemma}

\begin{proof}
If a move from position $A$ to position $B$ exists then $A \succ B$. If in addition, this is a Type II move, the total number of tokens is decreased by a multiple of $m$ implying $|A| \equiv |B| \pmod{m}$.

On the other hand, suppose $|A| \equiv |B| \pmod{m}$ and $A \succ B$. Let $A = (a_1,a_2,\ldots,a_n)$ and $B = (b_1,b_2,\ldots,b_n)$. From the $i^\text{th}$ heap in $A$, take away $a_i - b_i$ tokens. Because $A \succ B$, we have that $a_i - b_i \geq 0$ implying that the move is well-defined. Moreover, the total number of tokens removed must be divisible by $m$ as $|A| \equiv |B| \pmod{m}$.
\end{proof}

\section{$2$ Heap $m$-Modular Nim for Odd $m$}\label{2HeapOdd}

Rather than dealing with any number of heaps, we will start with $2$ heaps and odd $m$.

\subsection{An Example: $m=3$.}

\begin{example}[$m = 3$]

The P-positions of $3$-Modular Nim are the ordered pairs:

\begin{center}
$(0,0),$ \\
$(1,1),$ \\
$(2,2).$
\end{center}

Let the specified set be $\mathcal{S}$. We can manually verify that no members of $\mathcal{S}$ are connected by a legal move. Thus, it suffices to show that a move from any position $(a,b) \notin \mathcal{S}$  to an element of $\mathcal{S}$ exists.

\begin{enumerate}
\item Suppose that $\min(a,b) < 3$, then a Type I move must exist. 
\item On the other hand, if $\min(a,b) \geq 3$, then $(a,b) \succ \mathcal{S}$. By Lemma \ref{DominateTypeII}, a Type II move must exist because all residue classes modulo $3$ are covered by $\mathcal{S}$.
\end{enumerate}
\end{example}

Figure~\ref{2HeapMod3} displays the P-positions in 3-Modular Nim on a coordinate grid. 

\begin{figure}[htbp]
\begin{center}
\begin{tikzpicture}[scale = 1]
{\filldraw[black] (0,0) circle (2pt) node[anchor=west] {(0,0)};
\filldraw[black] (1,1) circle (2pt) node[anchor=west] {(1,1)};
\filldraw[black] (2,2) circle (2pt) node[anchor=west] {(2,2)};}
{\draw[black, thick] (0,0) -- (2,2);}
\end{tikzpicture}
\end{center}
\caption{P-positions in $3$-Modular Nim.}
\label{2HeapMod3}
\end{figure}
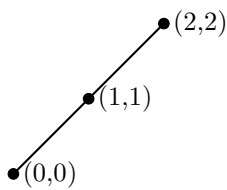

These patterns of P-positions for $m = 3$ suggests a similar structure exists for other integers. This motivates a generalization to all odd integers.

\begin{theorem}\label{Odd2HeapPPosition}For odd $m$, a position of $m$-Modular Nim with $2$ heaps is a P-position if and only it is of the form $(i,i)$ for integers $i$, where $0 \leq i < m$. \end{theorem}

\begin{proof}
Let the set of specified positions be $\mathcal{S}$. We first prove that no move exists between any of the specified positions. Because no distinct positions have a heap size in common, no Type I move exists. Suppose that a Type II move exists between distinct positions $(i,i)$ and $(j,j)$ where $0 \leq i,j < m$. By Lemma~\ref{DominateTypeII}, we must have that:
\[2i \equiv 2j \pmod{m}. \]

Because $\gcd(2,m) = 1$, we may divide both sides by 2:
\[i \equiv j \pmod{m}. \]

Because $0 \leq i,j < m$, we have that $i = j$, contradicting the assumption of distinctness. 

Now we must prove that for any position $(a,b) \notin \mathcal{S}$ there is a move to a position in $\mathcal{S}$. 

\begin{enumerate}
\item Suppose that $\min(a,b) < m$, then a Type I move must exist.
\item On the other hand, if $\min(a,b) \geq m$, then $(a,b) \succ \mathcal{S}$. Therefore, a Type II move must exist because all residue classes modulo $m$ are covered by $\mathcal{S}$.
\end{enumerate}
\end{proof}

\section{$2$ Heap $m$-Modular Nim for Any $m$}\label{2HeapEven}

\subsection{Another Example: $m=6$.}

We consider an example of $m$-Modular Nim, where $m$ is an even integer: $m = 6$.

\begin{example}[$m = 6$] \label{m=6}
We claim that the set of P-positions for $m = 6$ is the following set:

\begin{center}
$(0,0)$ \\
$(1,1)$ \\
$(2,2)$
\begin{align*}
(3,4) & \hspace{40pt} (4,3) \\
(5,6) & \hspace{40pt} (6,5) \\
(7,8) & \hspace{40pt} (8,7).
\end{align*}
\end{center}

Let the specified set be $\mathcal{S}$. As before, we can manually verify that no Type I or Type II moves connect any two members of $\mathcal{S}$. 

To show that $\mathcal{S}$ is the set of all P-positions, we must now show that for any position $(a,b) \notin \mathcal{S}$ there is a move to an element of $\mathcal{S}$.

\begin{enumerate}
\item If $\min(a,b) \leq 8$, then we may reach a member of $\mathcal{S}$ by removing the necessary number of tokens from the larger heap using a Type I move. 

\item On the other hand, if $\min(a,b) > 8$, we have that $(a,b) \succ \mathcal{S}$. Because $\mathcal{S}$ contains positions with each possible total heap-sum modulo $6$, there must exist a Type II move.
\end{enumerate}
\end{example}

We can see the existence of three distinct groups of P-positions for the case where $m = 6$. The first group has both heaps equal in size and is the same as P-positions for 3-Modular Nim. The second and third group can be viewed as the P-positions of 3-Modular Nim shifted by $(3,4)$ and $(4,3)$ respectively. 

This idea is further elucidated by Figure~\ref{2HeapMod6}. The red sections indicate the locations of the second and third groups that are shifted replicas of Figure~\ref{2HeapMod3}. This nesting structure is essential in finding a formula for the set of P-positions.

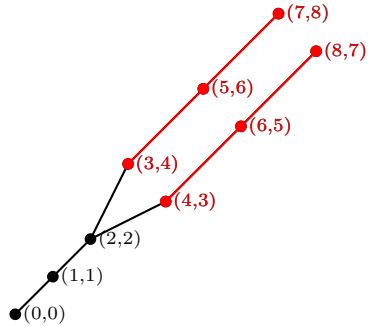
\begin{figure}[htbp]
\begin{center}
\begin{tikzpicture}[scale = 0.5]
\scriptsize
\filldraw[black] (0,0) circle (4pt) node[anchor=west] {(0,0)};
\filldraw[black] (1,1) circle (4pt) node[anchor=west] {(1,1)};
\filldraw[black] (2,2) circle (4pt) node[anchor=west] {(2,2)};
\draw[black, thick] (0,0) -- (2,2);
{
\filldraw[black] (3,4) circle (4pt) node[anchor=west] {(3,4)};
\filldraw[black] (4,3) circle (4pt) node[anchor=west] {(4,3)};
\filldraw[black] (5,6) circle (4pt) node[anchor=west] {(5,6)};
\filldraw[black] (6,5) circle (4pt) node[anchor=west] {(6,5)};
\filldraw[black] (7,8) circle (4pt) node[anchor=west] {(7,8)};
\filldraw[black] (8,7) circle (4pt) node[anchor=west] {(8,7)};
\draw[black, thick] (2,2) -- (3,4);
\draw[black, thick] (2,2) -- (4,3);
\draw[black, thick] (3,4) -- (7,8);
\draw[black, thick] (4,3) -- (8,7);

{
\filldraw[red] (3,4) circle (4pt) node[anchor=west] {(3,4)};
\filldraw[red] (4,3) circle (4pt) node[anchor=west] {(4,3)};
\filldraw[red] (5,6) circle (4pt) node[anchor=west] {(5,6)};
\filldraw[red] (6,5) circle (4pt) node[anchor=west] {(6,5)};
\filldraw[red] (7,8) circle (4pt) node[anchor=west] {(7,8)};
\filldraw[red] (8,7) circle (4pt) node[anchor=west] {(8,7)};
\draw[red, thick] (3,4) -- (7,8);
\draw[red, thick] (4,3) -- (8,7);
}

}
\end{tikzpicture}

\end{center}
\caption {P-positions for $6$-Modular Nim.}
\label{2HeapMod6}
\end{figure}

To make this nesting pattern more clear, we display the set of P-positions in $12$-Modular Nim in Figure~\ref{2HeapMod12}. The red section indicates one of the two embedded shifted replicas of P-positions in 6-Modular Nim as shown in Figure~\ref{2HeapMod6} whereas the green section shows the doubly nested and shifted copy of P-positions in $3$-Modular Nim. Note that because we are chiefly concerned with the recursive structure, the individual labels have been removed.

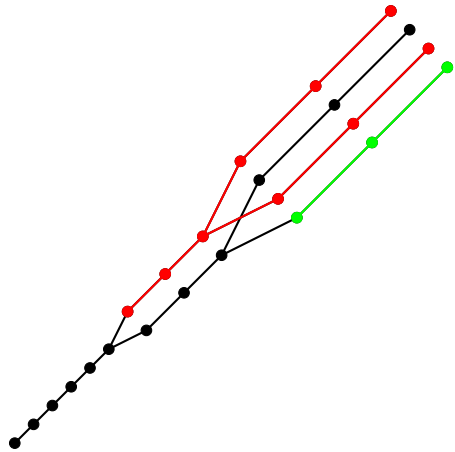
\begin{figure}[htbp]
\centering
\begin{tikzpicture}[scale = 0.25]
\filldraw[black] (0,0) circle (8pt) ;
\filldraw[black] (1,1) circle (8pt) ;
\filldraw[black] (2,2) circle (8pt) ;
\filldraw[black] (3,3) circle (8pt) ;
\filldraw[black] (4,4) circle (8pt) ;
\filldraw[black] (5,5) circle (8pt) ;
\filldraw[black] (6,7) circle (8pt) ;
\filldraw[black] (7,6) circle (8pt) ;
\filldraw[black] (8,9) circle (8pt) ;
\filldraw[black] (9,8) circle (8pt) ;
\filldraw[black] (10,11) circle (8pt) ;
\filldraw[black] (11,10) circle (8pt) ;
\filldraw[black] (12,15) circle (8pt) ;
\filldraw[black] (13,14) circle (8pt) ;
\filldraw[black] (14,13) circle (8pt) ;
\filldraw[black] (15,12) circle (8pt) ;
\filldraw[black] (16,19) circle (8pt) ;
\filldraw[black] (17,18) circle (8pt) ;
\filldraw[black] (18,17) circle (8pt) ;
\filldraw[black] (19,16) circle (8pt) ;
\filldraw[black] (20,23) circle (8pt) ;
\filldraw[black] (21,22) circle (8pt) ;
\filldraw[black] (22,21) circle (8pt) ;
\filldraw[black] (23,20) circle (8pt) ;
\draw[black, thick] (0,0) -- (5,5);
\draw[black, thick] (5,5) -- (6,7);
\draw[black, thick] (5,5) -- (7,6);
\draw[black, thick] (6,7) -- (10,11);
\draw[black, thick] (7,6) -- (11,10);
\draw[black, thick] (10,11) -- (12,15);
\draw[black, thick] (10,11) -- (14,13);
\draw[black, thick] (11,10) -- (15,12);
\draw[black, thick] (11,10) -- (13,14);
\draw[black, thick] (12,15) -- (20,23);
\draw[black, thick] (13,14) -- (21,22);
\draw[black, thick] (14,13) -- (22,21);
\draw[black, thick] (15,12) -- (23,20);

{
\filldraw[red] (6,7) circle (8pt) ;
\filldraw[red] (8,9) circle (8pt) ;
\filldraw[red] (10,11) circle (8pt) ;
\filldraw[red] (12,15) circle (8pt) ;
\filldraw[red] (14,13) circle (8pt) ;
\filldraw[red] (16,19) circle (8pt) ;
\filldraw[red] (18,17) circle (8pt) ;
\filldraw[red] (20,23) circle (8pt) ;
\filldraw[red] (22,21) circle (8pt) ;
\draw[red, thick] (6,7) -- (10,11);
\draw[red, thick] (10,11) -- (12,15);
\draw[red, thick] (10,11) -- (14,13);
\draw[red, thick] (12,15) -- (20,23);
\draw[red, thick] (14,13) -- (22,21);
}

>{
\filldraw[green] (15,12) circle (8pt) ;
\filldraw[green] (19,16) circle (8pt) ;
\filldraw[green] (23,20) circle (8pt) ;
\draw[green, thick] (15,12) -- (23,20);
}

\end{tikzpicture}
 \caption{P-positions for 12-Modular Nim.}
\label{2HeapMod12}
\end{figure}

\subsection{Potential P-positions}\label{sec:potentialpositions}

In order to formalize the previous notion of nesting, we will recursively define the positions that we later prove to be P-positions for any value of $m$.

\begin{definition}  
\[\mathcal{A}_m = \left \{ 
\begin{array}{ll}
\hspace{5pt} (i,i) \text{ where } 0 \leq i < m & \hspace{112pt} : m \text{ is odd} \\
\hspace{5pt} (i,i) \text{ where } 0 \leq i < \frac{m}{2} &  \hspace{112pt}: m \text{ is even.}
\end{array}
\right. \]
We will refer to any member of $\mathcal{A}_m$ as a \emph{trunk position}. 
\end{definition}

Note that the P-positions in $6$-Modular Nim (from Example \ref{m=6}) that appear before the ``splitting'' are trunk positions.

\begin{definition}
\[ \mathcal{B}_m = \left \{ 
\begin{array}{ll}
\hspace{5pt} \varnothing &: m \text{ is odd} \\
\begin{array}{l}\left(2a + \frac{m}{2} + 1, 2b + \frac{m}{2}\right) \\ \left(2a + \frac{m}{2}, 2b + \frac{m}{2} + 1\right) \end{array} 
\text{ where } (a,b) \in  \mathcal{B}_{\frac{m}{2}}\cup \mathcal{A}_{\frac{m}{2}} &: m \text{ is even.}
\end{array}
\right. \]
We will refer to any member of $\mathcal{B}_m$ as a \emph{branch position}.
\end{definition}

Note that the P-positions in $6$-Modular Nim (from Example \ref{m=6}) that appear after the ``splitting'' are branch positions. Also note that both positions in $\mathcal{B}_m $ generated from $(a,b)$ have the same sum of coordinates.

We seek to prove that the set of P-positions in $m$-Modular Nim is $\mathcal{Q}_{m} = \mathcal{A}_{m} \cup \mathcal{B}_{m}$.

\begin{definition}  We call the set of positions $\mathcal{Q}_{m}$ \emph{potential positions}. \end{definition} 

Note that the recursion that builds $\mathcal{Q}_{2m}$ from $\mathcal{Q}_{m}$ allows us to provide a recursion for the number of elements in $\mathcal{Q}_{m}$:
\[|\mathcal{Q}_{2m}| = m+ 2|\mathcal{Q}_{m}|.
\]

Because of the recursive doubling involved in generating trunk and branch positions, it is natural to consider the number of times that this doubling can occur. In other words, we wish to count the greatest power of $2$ that divides $m$.

\begin{definition}Define the \emph{$2$-adic order} of $m$ as the highest power of $2$ that divides $m$. We denote this arithmetic function as $\nu_2(m)$ \cite{NumberTheory}.\end{definition}

Now we prove a series of preliminary results. 
The following lemma shows that a number may appear in a potential position in a particular coordinate at most once. 

\begin{lemma} \label{2HeapTypeI} 
If $(a,b)$ and $(a,c) \in \mathcal{Q}_m$, then $b = c$. \end{lemma}

\begin{proof} 
We proceed by induction on the $2$-adic order of $k$. For the base case, suppose that $k$ is odd. From the definition of $\mathcal{Q}_k$, where $k$ is odd, no two distinct positions share a heap size, because all positions are of the form $(i,i)$ where $0 \leq i < k$.

For the inductive step, suppose that the lemma is true for $\mathcal{Q}_k$. We wish to prove that if $(a,b)$ and $(a,c) \in \mathcal{Q}_{2k}$, then $b = c$. We will split this into cases:

\begin{enumerate}
\item Suppose that both $(a,b)$ and $(a,c)$ are trunk positions. No two trunk positions share a heap size because they are of the form $(i,i)$ for different $i$.
\item Suppose one of the positions $(a,b)$ and $(a,c)$ is a trunk position and the other is a branch position. Any trunk position in $\mathcal{Q}_k$ is strictly dominated by $(k,k)$ and any branch position strictly dominates $(k,k)$. Therefore, a branch position and a trunk position cannot share a coordinate.
\item Suppose that both $(a,b)$ and $(a,c)$ are elements of $\mathcal{B}_{2k}$. If $a$ has the same parity as $k$, then $a = 2a^\prime + k$. Thus, we may write $b = 2b^\prime + k + 1$ and $c = 2c^\prime + k + 1$ where $\left(a^\prime, b^\prime\right)$ and $\left(a^\prime, c^\prime\right)$  are in $\mathcal{Q}_{k}$. By the inductive hypothesis, $b^\prime = c^\prime$. Therefore, $b = c$.

A similar symmetric argument works when $a$ has the opposite parity as $k$.
\end{enumerate}
Thus, we are done by induction. 
\end{proof}

We can strengthen Lemma~\ref{2HeapTypeI} by showing that the set of integers allowed to be a coordinate of a potential position only consists of consecutive numbers.

\begin{lemma}\label{2HeapPPositionCount}
Integers between $0$ and $k\left(\frac{\nu_2(k)}{2} + 1\right) -1$ inclusive form the set of all first coordinates in the set of potential positions $\mathcal{Q}_k$. By symmetry the same is true for the second coordinate.\end{lemma}

\begin{proof}
As before, we induct on the $2$-adic order of $k$. Our base case is the odd integers. In Theorem \ref{Odd2HeapPPosition}, we defined the set of potential positions for odd integers $k$ to be all $(i,i)$ where $0 \leq i < k$. This concludes the base case, as $\nu_2(k)=0$.

For the inductive step, consider $\mathcal{Q}_{2k}$. For trunk positions, the first coordinate ranges over all numbers from $0$ to $k - 1$ inclusive. For branch positions, it ranges over all numbers of the form $2a+k$ and $2a+k+1$, where $a$ is a first coordinate of an element of $\mathcal{Q}_k$. 

By induction, $a$ is between 0 and $k\left(\frac{\nu_2(k)}{2} + 1\right) -1$ inclusive. Thus the new branch positions for $2k$ have first coordinates in the range between $k$ and 

\begin{align*}
2\left(k\left(\frac{\nu_2(k)}{2} + 1\right) -1 \right) + k + 1 &= \\
2k\left(\frac{\nu_2(k)}{2} + 1\right) + k - 1 &= \\
2k\left(\frac{\nu_2(k)}{2} + 1 + \frac{1}{2} \right) -1 &= \\
2k\left(\frac{\nu_2(2k)}{2} + 1  \right) -1.
\end{align*}

Therefore, an integer is between $0$ and $2k\left(\frac{\nu_2(k)}{2} + 1\right) -1$ inclusive if and only if it appears as a first coordinate in the set of potential positions $\mathcal{Q}_{2k}$, completing the induction.
\end{proof}

\begin{corollary}\label{thm:totalcount}
\[|\mathcal{Q}_m| = m\left(\frac{\nu_2(m)}{2} + 1\right).\] \end{corollary}

With Lemma~\ref{2HeapTypeI} and Lemma~\ref{2HeapPPositionCount}, we can now explicitly determine the number of potential positions.

The following lemma shows that two potential positions that share a remainder modulo $m$ have the same sum.

\begin{lemma}\label{2HeapTypeII} If $(a,b)$ and $(c,d) \in \mathcal{Q}_k$ and $a + b \equiv c + d \pmod{k}$, then $a + b = c + d$. \end{lemma}

\begin{proof}
We proceed by induction on $\nu(k)$. For the base case, suppose that $k$ is odd. All potential positions are of the
from $(i,i)$ where $0 \leq i < k$. Let $(a,b) = (m,m)$ and $(c,d) = (n,n)$. We can rewrite our condition as:
\[2m \equiv 2n \pmod{k}. \]

Because $\gcd(2,k) = 1$, we may divide both sides by $2$:
\[m \equiv n \pmod{k}. \]

Because $0 \leq m,n < k$, we have that $m = n$. Thus, $a + b = c + d$ where $k$ is odd. 

For the inductive hypothesis, assume the lemma is true for $k$. We seek to now prove the corresponding statement for $2k$. We will divide this into cases:

\begin{enumerate}
\item Suppose that both $(a,b)$ and $(c,d)$ are trunk positions in $\mathcal{Q}_{2k}$. Because trunk positions are of the from $(i,i)$ where $0 \leq i < k$, we have that  $(a,b) = (m,m)$ and $(c,d) = (n,n)$ where $0 \leq m,n < k$. We can rewrite our condition as:
\[2m \equiv 2n \pmod{2k}. \]

Because $\gcd(2,2k) = 2$, we may divide both sides by $2$ only if we divide the modular base by $2$ as well:
\[m \equiv n \pmod{k}. \]

Because $0 \leq m,n < k$, this implies that $m = n$, finishing this case. 

\item Suppose one of $(a,b)$ and $(c,d)$ is a trunk position and the other is a branch position in $\mathcal{Q}_{2k}$. The heap-sum of a trunk position is even and the heap-sum of a branch position is odd. Therefore, they cannot have the same remainder modulo an even number $2k$.

\item Suppose that both $(a,b)$ and $(c,d)$ are branch positions in $\mathcal{Q}_{2k}$. Therefore,  $(a,b)$ is equal to $(2a^\prime + k, 2b^\prime + k + 1)$ or $(2a^\prime + k + 1, 2b^\prime + k)$ where $\left(a^\prime, b^\prime\right)$ is a position in $\mathcal{Q}_{2k}$. In either case, 
\[a + b = 2a^\prime + 2b^\prime + 2k + 1.\]

Similarly,
\[c + d = 2c^\prime + 2d^\prime + 2k + 1.\]

with $\left(c^\prime, d^\prime\right)$ in $\mathcal{Q}_{2k}$. We may now rewrite our given condition 
\[a + b \equiv c + d \pmod{2k}\]
as 
\[2a^\prime + 2b^\prime \equiv 2c^\prime + 2d^\prime \pmod{2k}.\]

Because $\gcd(2,2k) = 2$, we may divide both sides by $2$ only if we divide the modular base by $2$ as well:
\[ a^\prime + b^\prime \equiv c^\prime + d^\prime \pmod{k}. \]

By the inductive hypothesis, $a^\prime + b^\prime = c^\prime + d^\prime$. Therefore, $a + b = c + d$.

Thus, we are done by induction.
\end{enumerate}
\end{proof}

\begin{corollary}\label{cor:trunkremainders}
Trunk positions in $\mathcal{Q}_k$ have distinct remainders modulo $k$.
\end{corollary}

We are now ready to prove that every position not in $\mathcal{Q}_m$ has a move to a position in $\mathcal{Q}_m$ in the game of $m$-Modular Nim.

\begin{lemma}\label{MovesFromNotQtoQ}Every position not in $\mathcal{Q}_m$ has a move to an element of $\mathcal{Q}_m$ in the game of $m$-modular Nim.\end{lemma}

\begin{proof}
Let there be a position $(a,b) \notin \mathcal{Q}_k$. We will as usual induct on the $2$-adic order. Our base case is when $k$ is odd. By Lemma~\ref{Odd2HeapPPosition}, $\mathcal{Q}_k$ is the complete set of P-positions in $k$-Modular Nim. Therefore, all other positions are N-positions. From the definition of an N-position, every N-position has a move to a P-position. 

For the inductive step, assume that this lemma is true for $\mathcal{Q}_k$ in the game of $k$-Modular Nim. We want to show the corresponding statement for $\mathcal{Q}_{2k}$ in the game of $2k$-Modular Nim. Again we will use casework. Assume without loss of generality that $a \leq b$.

\begin{enumerate}
\item Suppose that $0 \leq a < k$. Then, we have $(a,a) \in \mathcal{Q}_{2k}$. Therefore, there is a Type I move from $(a,b)$ to $(a,a)$, a member of $\mathcal{Q}_{2k}$.
\item Suppose that $k \leq a \leq b$ and $a \equiv b \pmod{2}$. The heap-sums of trunk positions in $\mathcal{Q}_{2k}$ form a set of $k$ even integers. By Corollary~\ref{cor:trunkremainders} there exists a trunk position with any given even remainder modulo $2k$. Therefore, there exists a trunk position with the same even remainder modulo $2k$ as $a+b$. As $(a,b)$ dominates $(k,k)$, it also strictly dominates any trunk position. Therefore by Lemma \ref{DominateTypeII}, there exists a Type II move from $(a,b)$ to the trunk position with the same remainder.

\item Suppose that $k \leq a \leq b$ and $a \equiv k \pmod{2}$ while $b \equiv k + 1 \pmod{2}$. Consider the position:
\[ \left(a^\prime, b^\prime\right) = \left(\frac{a - k}{2}, \frac{b - k - 1}{2} \right) \]

Note that $\left(a^\prime, b^\prime\right)$ is not an element of $\mathcal{Q}_k$, as $\left(2a^\prime + k, 2b^\prime + k + 1\right) = (a,b)$ is not an element of $\mathcal{Q}_{2k}$. Therefore, by the inductive hypothesis there exists a position $\left(q_1,q_2\right) \in \mathcal{Q}_{k}$ that can be reached with a Type II move from $\left(a^\prime, b^\prime\right)$ in the game of $k$-Modular Nim. 

Thus, there exists a Type II move from $(a,b) = \left(2a^\prime + k, 2b^\prime + k + 1\right)$ to $\left(2q_1 + k,2q_2 + k + 1\right)$ in the game of $2k$-Modular Nim. Note that $\left(2q_1 + k,2q_2 + k + 1\right)$ is an element of $\mathcal{Q}_{2k}$. This finishes this case.

\item Suppose that $k \leq a \leq b$ and $a \equiv k + 1\pmod{2}$ while $b \equiv k  \pmod{2}$. We are done by a similar symmetric argument to the previous case.
\end{enumerate}
This completes the induction. \end{proof}

These results enable us to determine a recursive definition of P-positions in $m$-Modular Nim with $2$ heaps.

\begin{theorem}If $\mathcal{P}_m$ is the set of P-positions for $m$-Modular Nim with $2$ heaps, then $\mathcal{P}_m = \mathcal{Q}_m$. \end{theorem}

\begin{proof}We must first prove that no move exists between any two elements of $\mathcal{Q}_m$. Let $(a,b)$ and $(c,d)$ be distinct elements of $\mathcal{Q}_m$. If a Type I move exists between $(a,b)$ and $(c,d)$, we may assume without loss of generality that $a = c$. Thus $(a,b)$ and $(a,d)$ are distinct elements of $\mathcal{Q}_m$. However, from Lemma \ref{2HeapTypeI}, then we have that $b = d$ which contradicts distinctness. 

Now suppose that a Type II move exists. By Lemma \ref{DominateTypeII}, we have that $a + b \equiv c + d \pmod{m}$. However, from Lemma \ref{2HeapTypeII}, we have that $a + b = c + d$. However, every move strictly decreases heap-sum, so no such move exists.

Moreover by Lemma \ref{MovesFromNotQtoQ}, every position not in $\mathcal{Q}_m$ has a move to an element of $\mathcal{Q}_m$. Therefore, $\mathcal{P}_m = \mathcal{Q}_m$. \end{proof}

We are now ready to describe the set of P-positions explicitly.

\section{Explicit Description of P-positions}\label{sec:explicit}

Let $m = k \cdot 2^n$ where $k$ is odd. The P-positions in the game of $m$-Modular Nim are built recursively from the positions of $k$-Modular Nim, by using the recursion described in Section~\ref{sec:potentialpositions}. The recursion is used $n$ times. 

The recursion procedure is similar for all $m$ with the same 2-adic order. Figure~\ref{fig:8mod} depicts P-positions for 8-Modular Nim to emphasize the branching.

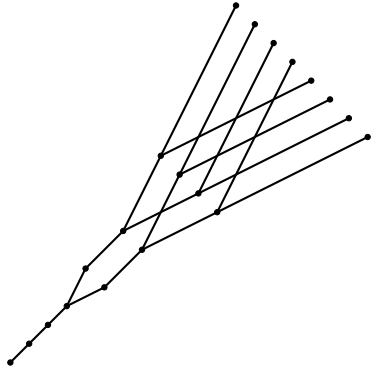
\begin{figure}[htbp]
\begin{center}
\begin{tikzpicture}[scale = 0.25]
\scriptsize
\filldraw[black] (0,0) circle (4pt);
\filldraw[black] (1,1) circle (4pt);
\filldraw[black] (2,2) circle (4pt);
\draw[black, thick] (0,0) -- (3,3);
{
\filldraw[black] (3,3) circle (4pt);
\filldraw[black] (4,5) circle (4pt);
\filldraw[black] (5,4) circle (4pt);
\draw[black, thick] (3,3) -- (4,5);
\draw[black, thick] (3,3) -- (5,4);

\filldraw[black] (6,7) circle (4pt);
\filldraw[black] (7,6) circle (4pt);

\filldraw[black] (8,11) circle (4pt);
\filldraw[black] (9,10) circle (4pt);
\filldraw[black] (10,9) circle (4pt);
\filldraw[black] (11,8) circle (4pt);

\filldraw[black] (12,19) circle (4pt);
\filldraw[black] (13,18) circle (4pt);
\filldraw[black] (14,17) circle (4pt);
\filldraw[black] (15,16) circle (4pt);

\filldraw[black] (19,12) circle (4pt);
\filldraw[black] (18,13) circle (4pt);
\filldraw[black] (17,14) circle (4pt);
\filldraw[black] (16,15) circle (4pt);

\draw[black, thick] (5,4) -- (7,6);
\draw[black, thick] (4,5) -- (6,7);

\draw[black, thick] (7,6) -- (11,8);
\draw[black, thick] (7,6) -- (9,10);
\draw[black, thick] (6,7) -- (8,11);
\draw[black, thick] (6,7) -- (10,9);

\draw[black, thick] (8,11) -- (12,19);
\draw[black, thick] (8,11) -- (16,15);
\draw[black, thick] (11,8) -- (19,12);
\draw[black, thick] (11,8) -- (15,16);
\draw[black, thick] (9,10) -- (13,18);
\draw[black, thick] (9,10) -- (17,14);
\draw[black, thick] (10,9) -- (18,13);
\draw[black, thick] (10,9) -- (14,17);
}
\end{tikzpicture}

\end{center}
\caption {P-positions for $8$-Modular Nim.}
\label{fig:8mod}
\end{figure}

\begin{definition}\label{Branch Level}The \textit{$i$-level branch} is the set of P-positions that are generated from the trunk of $\mathcal{Q}_{k\cdot 2^{n-i}}$ using $i$ splitting procedures as described in Section~\ref{sec:potentialpositions}.\end{definition}

The P-positions in the $i$-level branch are of the form 
\[(2^ia + f_im + b_1, 2^ia + f_im + b_2),\]
where $f_i$ is the coefficient by which $m$ is multiplied after completing $i$ splitting procedures. Moreover, we have $0 \leq a < k \cdot 2^{n-i-1}$, for $i < n$, and $0 \leq a < k$, for $i=n$.

Let us trace through recursion and find $f_i$. The $i$-level branch is recursively generated from the trunk of $\mathcal{Q}_{\frac{m}{2^i}}$. After the first recursion the coefficient is $\frac{m}{2^i}$. After the next recursion it is \[2\cdot \frac{m}{2^i} + \frac{m}{2^{i-1}}=2 \cdot \frac{m}{2^{i-1}}.\] Repeating again we get \[2\cdot \frac{m}{2^{i-2}} + \frac{m}{2^{i-2}}=3 \cdot \frac{m}{2^{i-2}}.\] As we continue, we see that $f_i = \frac{i}{2}$. 

We also can calculate that the largest possible value for $b_1$ is $2^i-1$, and the smallest is 0. The same is true for $b_2$. Therefore, the following lemma follows.

\begin{lemma}
The P-positions within the $i$-level branch correspond to all possible values for $0 \leq b_1 < 2^i$ and $b_2 = 2^i-1-b_1$.
\end{lemma}

\begin{proof}
The P-positions corresponding to the same $a$ within the $i$-level branch have the same heap-sum. They all have different first coordinates and there are $2^i$ of them.\end{proof}

Combining all the results together, we deduce the following theorem.

\begin{theorem}
The P-positions of $m$-Modular Nim are of the form
\[\left(2^ia + \frac{im}{2} + b, 2^ia + \frac{im}{2} + 2^i-1-b\right),\]
where $0 \leq i \leq n$ and $0 \leq b < 2^i$. In addition, $0 \leq a < k \cdot 2^{n-i-1}$, for $i < n$, and $0 \leq a < k$, for $i = n$.
\end{theorem}

Note that 0-level branch corresponds to the trunk and the formula correctly produces positions of the form:
\[(a,a)\]
for $0 \leq a < k \cdot 2^n = m$. Similarly, the 1-level branch positions are of the form: 
\[\left(2a + \frac{m}{2} + 1, 2a + \frac{m}{2}\right) \text{ and } \left(2a + \frac{m}{2}, 2a + \frac{m}{2}+1\right),\]
for $0 \leq a < k \cdot 2^{n-1}$. 

We see that in the list of all P-positions the range of the coordinates with the higher heap-sum is above the range of the coordinates with the lower heap-sum. 

\begin{corollary}
Any P-position with a greater sum dominates every P-position with a lower sum.
\end{corollary}

In addition, on level $i$ the first coordinate ranges from $f_im=i\frac{m}{2}$ to \[f_im+\frac{m}{2}-1=(i+1)\frac{m}{2}-1 = f_{i+1}m- 1.\]
This further confirms Lemma~\ref{2HeapPPositionCount} in that the numbers in the first and second coordinates of P-positions are consecutive. We can also see that the largest number in the range is $2^nk+n\frac{m}{2}-1$, which, not surprisingly, matches Corollary~\ref{thm:totalcount}.

The description above allows us to calculate the second coordinate of the P-position, given the first coordinate. 

\begin{lemma}
Suppose the first coordinate of an $m$-Modular Nim P-position is $x$, then the second coordinate is
\[ 2^{\left \lfloor \frac{2x}{m}\right  \rfloor+1} \left \lfloor \frac{x-\left \lfloor \frac{2x}{m}\right  \rfloor \frac{m}{2}}{2^i}\right \rfloor + \left \lfloor \frac{2x}{m}\right \rfloor m + 2^{\left \lfloor \frac{2x}{m}\right \rfloor}-1 -x.\]
\end{lemma}

\begin{proof}
Suppose the first coordinate is $x$, then the level $i$ is defined as $\left\lfloor \frac{2x}{m}\right\rfloor$. After that we can calculate $a$ in the formula as  $\left\lfloor \frac{x-im/2}{2^i}\right\rfloor$. Therefore, the second coordinate is:
\[ 2^{i+1}a + im + 2^i-1 -x.\]\end{proof}

\section{$m$-Modular Nim for Odd $m$ and any Number of Heaps}\label{sec:manyheaps}

We will now generalize the results from  Theorem~\ref{Odd2HeapPPosition} to any number of heaps.

\begin{theorem}Let $a = (a_1,a_2,\ldots,a_n)$ be a position in $m$-Modular Nim with $n$ heaps, where $m$ is odd. It is a P-position if and only if:

\begin{enumerate}
\item $|a| < 2m$
\item $\oplus_i a_i = 0$.
\end{enumerate}\end{theorem}

\begin{proof}
Let the specified set be $\mathcal{Q}$. We will first show that no move exists between any two positions in $\mathcal{Q}$. By Bouton's Theorem, no Type I move exists. 

Because the heap-sum of all members of $\mathcal{Q}$ is less than $2m$, any Type II move must subtract exactly $m$ tokens total. However, all members of $\mathcal{Q}$ have an even heap-sum as the bitwise XOR of their heap sizes is $0$. Because $m$ is odd, no two elements of $\mathcal{Q}$ can be connected by a Type II move.

We must now show that every position $p=(p_1,p_2,\ldots,p_n)$ not in $\mathcal{Q}$ has a move to an element of $\mathcal{Q}$. 

\begin{enumerate}
\item Suppose that $p$ is not a P-position in Nim and there is a Nim-move from $p$ to a position $p^\prime$ such that $|p^\prime| < 2m$. Then $p^\prime \in \mathcal{Q}$ and we have found our move.
\item Suppose that $p$ is not a P-position in Nim; and there is a Nim-move from $p$ to a position $q$ such that $|q| \geq 2m$. There exists a number $k$, such that $|p| -km < 2m$ and $|p| -km$ is even. In the game of Nim, there exists an optimal play in which only $1$ token is removed per turn. That means, for any P-position $q$, there exists a P-position dominated by $q$ for every even heap-sum less than $|q|$. In particular, there exists a P-position in Nim, $p^\prime$, such that $|p^\prime| = |p| -km$ and $q$ dominates $p^\prime$. Therefore, $p$ dominates $p^\prime$, and $p^\prime$ is reachable from $p$ by a Type II move.
\item Suppose that $p$ is a P-position in Nim, and $|p| \geq 2m$. A similar argument to the one above shows that there exists a P-position of Nim $p^\prime \in \mathcal{Q}$ that is reachable from $p$ via a Type II move.
\end{enumerate} \end{proof}

\section{Acknowledgements}
The authors are grateful to the MIT-PRIMES program for supporting this research and to Prof. Aviezri Fraenkel for helpful suggestions.

\end{document}